\documentclass{article} 
\usepackage{mathmag_copy}  
\makeatletter
\def\ps@headings{%
  \def\@oddhead{}%
  \def\@evenhead{}%
  \def\@oddfoot{\hfill\thepage\hfill}%
  \def\@evenfoot{\hfill\thepage\hfill}%
}
\makeatother
\usepackage{cite}
\usepackage{comment}
\usepackage{cite} 

\theoremstyle{theorem}
\newtheorem{theorem}{Theorem}
\newtheorem{proposition}{Proposition}

\theoremstyle{definition}
\newtheorem*{definition}{Definition}

\newtheorem{example}{Example}
\newtheorem{conjecture}{Conjecture}

\allowdisplaybreaks
\author{David McCune\\               
\scriptsize Division of Analytical Sciences\\      
William Jewell College\\                  
\url{mccuned@william.jewell.edu}\\                       
}                        

\makeatletter
\@addtoreset{footnote}{page}
\makeatother

\begin{document}

\title{A Geometric View of Some Apportionment Paradoxes, Revisited}

\maketitle

\noindent  One of my favorite \emph{Mathematics Magazine} articles is ``A Geometric View of Some Apportionment Paradoxes'' by Brent Bradberry, published in 1992. I love the exposition, and the marriage of geometry and apportionment theory. I admit I am biased because I have done some research in apportionment theory, but I think most readers would appreciate Bradberry's exploration of apportionment paradoxes using simplicial geometry.

The paper does have one glaring flaw, however: its most substantive mathematical result is wrong. In Theorem 2 of \cite{Bradberry}, Bradberry provides an incorrect formula for the probability that the Alabama paradox occurs under Hamilton's apportionment method in the three-state case. In this article, we prove the correct result and investigate the origin of Bradberry's mistake, using this analysis as a springboard to explore the Alabama paradox more broadly.

\section{The Apportionment Problem}

Apportionment theory is one of the few branches of mathematics that can trace its origin to the US Constitution. Article I, Section 2 states, ``Representatives and direct taxes shall be apportioned among the several states
which may be included in this Union, according to their representative number.'' This is an 18th-century way of saying that if a state has $p\%$ of the country's population then the state should have $p\%$ of the seats in the US House of Representatives. For example, as of this writing the US has approximately 342.6 million residents and my birth state of Texas has approximately 31.7 million residents, accounting for 9.3\% of the population. There are 435 seats in the US House, and thus Texas is due 40.46 seats. Each seat represents a human representative from Texas, and humans famously can't be fractional. The \emph{apportionment problem} is figuring out how to distribute whole seats when states are entitled to fractional amounts.

When I first present the apportionment problem to my students, I'm met with raised eyebrows and confusion tinged with a bit of pity. At first glance, there does not seem to be enough of a ``problem'' to justify a field of mathematics, because every 5th grader knows how to transform decimals into whole numbers. My students' reaction is something along the lines of: ``Has this professional mathematician not heard of rounding to the nearest whole number?''

To see why nearest integer rounding does not solve the problem, suppose a country with 10 seats in its House has three small states with populations 4400, 3400, and 2200. Then the states are due 4.4, 3.4, and 2.2 seats in the House, respectively, and standard rounding would allocate only 9 seats. We would then need to decide what to do with the one leftover seat. One work-around solution would be to change the size of the House to a number where the rounding works out, but for many reasons the size of the House is fixed in apportionment theory. In the US, for example, the House size is fixed at 435 by law, and changing this value at every Census to satisfy standard rounding does not seem desirable.

\section{Hamilton's Method and the Alabama Paradox}

The first known proposed solution to the apportionment problem is due to Alexander Hamilton in 1792, during the debate over the  first  apportionment in the US after the 1790 census. To define his method, we first establish our notation for the paper.

Suppose we have $S$ states with $H$ seats to apportion, and each state's proportion of the country's population is $p_i$ so that $\sum_{i=1}^S p_i =1$. Each state's \emph{quota}  is $q_i=p_i\cdot H$, where  $\sum_{i=1}^S q_i =H$. An \emph{apportionment} is a vector $\mathbf{a}=(a_1, \dots, a_S)$ of non-negative integers satisfying $\sum_{i=1}^S a_i=H$. Informally, an \emph{apportionment method} is a function whose input is a quota vector $\mathbf{q}=(q_1, \dots, q_S)$ and whose output is an apportionment. Equivalently, the input to an apportionment method could be a proportion vector $\mathbf{p}=(p_1, \dots, p_S)$ if we also input $H$. 

Alexander Hamilton proposed the following method for apportioning seats. Calculate each state's quota, then give each state $\lfloor q_i \rfloor$ seats. Unless all quotas are integers there is a nonzero number $k$ of remaining seats; allocate one seat each to the $k$ states whose fractional parts of their quotas are largest.

\begin{example}\label{first_ex}
Suppose a country has three states $A$, $B$, and $C$ with populations 5000, 3000, and 1000. Suppose $H=4$. Then the quotas are $(5000/9000)\cdot 4=2.22$, $(3000/9000)\cdot 4=1.33$, and $(1000/9000)\cdot 4=0.44$, respectively. Hamilton's method initially allocates 2, 1, and 0 seats to $A$, $B$, and $C$, leaving one seat unallocated. Since  $C$ has the largest fractional part of 0.44, this state receives the one extra seat, and Hamilton apportions 2 seats to $A$ and 1 each to $B$ and $C$ for an apportionment vector of $(2,1,1)$. 
\end{example}

What happens if we increase the number of seats from four to five in the previous example? The quotas become 2.78, 1.67, and 0.56, respectively. As before, Hamilton's method initially allocates two seats to $A$, one to $B$, and none to $C$. The two remaining seats go to $A$ and $B$, which have the largest fractional parts. Thus, the resulting apportionment is $(3,2,0)$. This is a surprising outcome: the populations of the three states have not changed, and the country has increased the total number of seats available for apportionment. We might reasonably expect that increasing the number of seats could only benefit the states, or at least leave their allocations unchanged. That is, an increase of House size should be similar to a rising tide lifting all boats. Instead, state $C$ loses its only seat. This phenomenon is known as the \emph{Alabama paradox}, named for the discovery in 1880 that Hamilton's method would award Alabama eight seats in a House of size 299, but only seven seats in a House of size 300. For a well-written history of apportionment and paradoxes in the US context, see the classic book \cite{BalinskiYoung}.

\begin{definition}
We say that a proportion vector $\mathbf{p}$
demonstrates the \emph{Alabama paradox} for
House size $H$ if some state receives fewer
seats under Hamilton's method with House
size $H+1$ than with House size $H$.
\end{definition}

This definition assumes $H$ increases by only one seat; we analyze the case of increasing $H$ by more than one later in the article.

\section{Geometry of Apportionment}

The above example shows that the Alabama paradox can occur, but it tells us little about how common the phenomenon is. For a fixed House size, which population distributions give rise to the paradox? What proportion of all possible population distributions exhibit this behavior? As Bradberry demonstrates, these questions can be investigated geometrically.

Following Bradberry \cite{Bradberry}, we restrict our attention to $S=3$. For a fixed House size $H$, the set of possible quota vectors is
\[
\Sigma^H=\{(q_1,q_2,q_3)\in\mathbb{R}_{\geq 0}^3:q_1+q_2+q_3=H\}.
\]
Geometrically, $\Sigma^H$ is an equilateral triangle with vertices $(H,0,0)$, $(0,H,0)$, and $(0,0,H)$. Each possible apportionment $\mathbf{a}=(a_1,a_2,a_3)$ corresponds to an integer lattice point in this triangle. The \textbf{apportionment region of} $\mathbf{a}$ is the set of quota vectors for which Hamilton's method produces the apportionment $\mathbf{a}$. 

Bradberry observes that Hamilton's method has a natural geometric interpretation: for a given quota vector $\mathbf{q}$, simply choose the apportionment vector $\mathbf{a}$ closest to $\mathbf{q}$ under the standard Euclidean metric. This gives an equivalent geometric definition of Hamilton's method. It also explains the shape of the apportionment regions. Bradberry shows that each region consists of the points closer to one apportionment vector than to any other, so the regions are regular hexagons centered at their corresponding lattice points, truncated by the edges of the simplex when necessary. In particular, apportionment vectors with all positive entries have regions that are full hexagons, while those with exactly one zero entry correspond to half-hexagons. The three corners of the simplex correspond to one-sixth slices of hexagons. The left panel of Figure \ref{figure:first_HAM} illustrates this hexagonal tiling for $H=4$. The boundaries between adjacent regions correspond to quota vectors for which Hamilton's method produces a ``tie,'' with two possible apportionments. For example, $(2.2, 0.4, 1.4)$ is a point on the intersection of the boundaries of the regions corresponding to $(2, 0, 2)$ and $(2,1,1)$. Thus, the apportionment is uniquely determined except on these boundaries.  

To study the Alabama paradox, we compare apportionment diagrams for different House sizes. Since the quotas depend on $H$, we scale each diagram to the unit simplex
\[
\Delta^2=\{(p_1,p_2,p_3)\in\mathbb{R}_{\geq 0}^3: p_1+p_2+p_3=1\}.
\]
Each point in this simplex represents a fixed population distribution, allowing us to compare the apportionments produced by Hamilton's method for different House sizes. By overlaying the diagrams for $H$ and $H+1$, we can identify the regions where increasing the House size causes a state to lose a seat. These are precisely the regions exhibiting the Alabama paradox. The right panel of Figure \ref{figure:first_HAM} illustrates this phenomenon for $H=4$, with the paradoxical regions shaded. Bradberry presents this same figure in his analysis of the paradox. Note that the population proportion vector $(5/9, 3/9, 1/9)$ from Example \ref{first_ex} lies in the lower left shaded region of the figure. The geometry lets us see that when $H=4$ this point lies in the apportionment region for $(2, 1, 1)$ but when $H=5$ it lies in the region for $(3,2,0)$. We refer to a shaded region as an \emph{overlap region}.

For larger values of $H$, the apportionment diagrams become much larger, making it useful to restrict  to a smaller portion of the simplex for some of our visualizations. Fortunately, the diagrams are symmetric under permutations of the three states. In particular, the six possible orderings of $p_1$, $p_2$, and $p_3$ divide the simplex into six congruent regions, each of which can be obtained from the others by reflection. Thus, we can restrict our attention to the region where $p_1\ge p_2\ge p_3$ (the lower-left sixth of the full diagram) without losing any information about the geometry of the paradox. Figure \ref{figure_AL_10} illustrates this portion of the simplex for $H=10$, with the $H=11$ diagram overlaid and the paradoxical regions shaded.

\begin{figure}
\begin{tabular}{cc}
\includegraphics[width=0.5\textwidth]{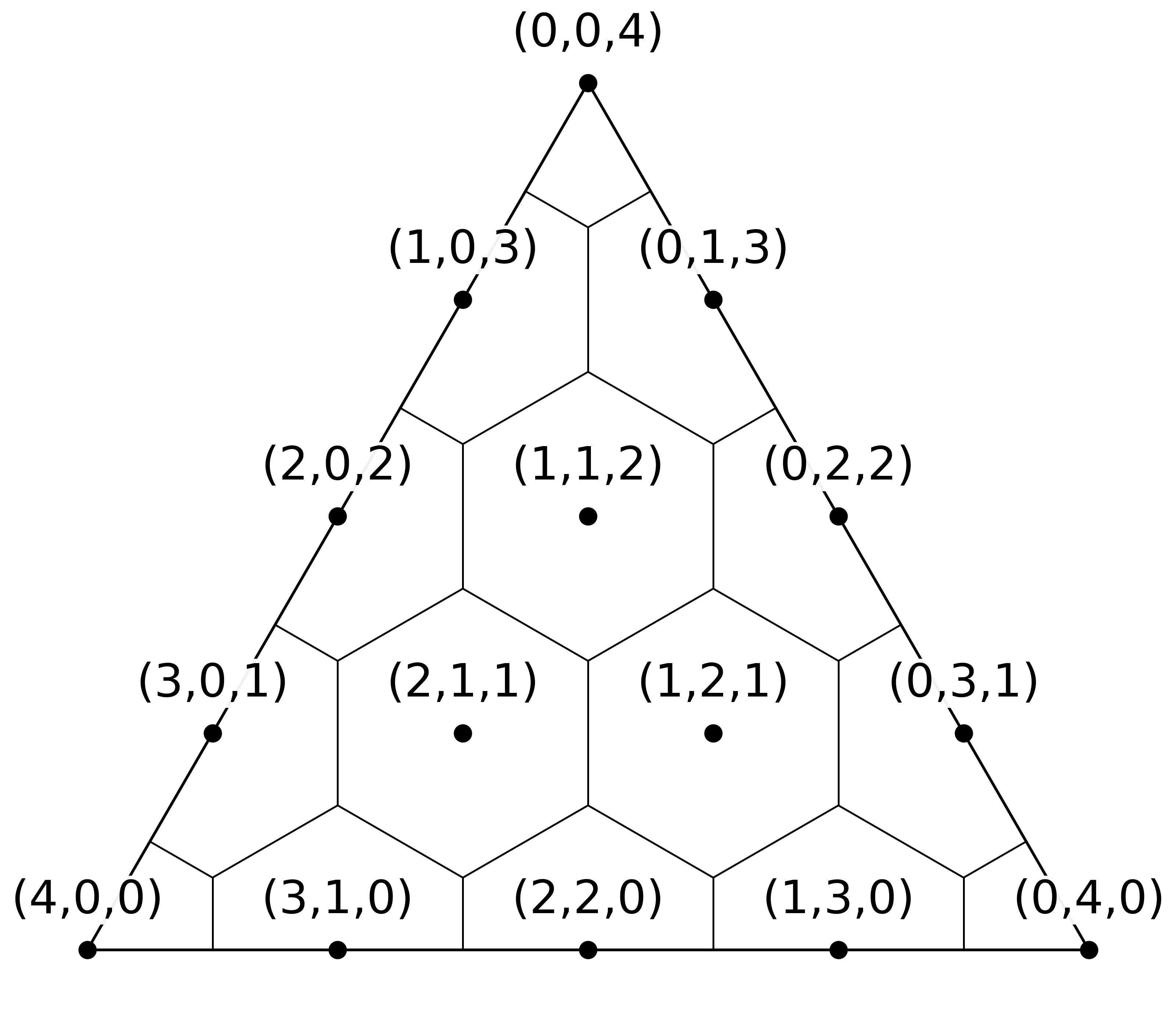}&\includegraphics[width=0.5\textwidth]{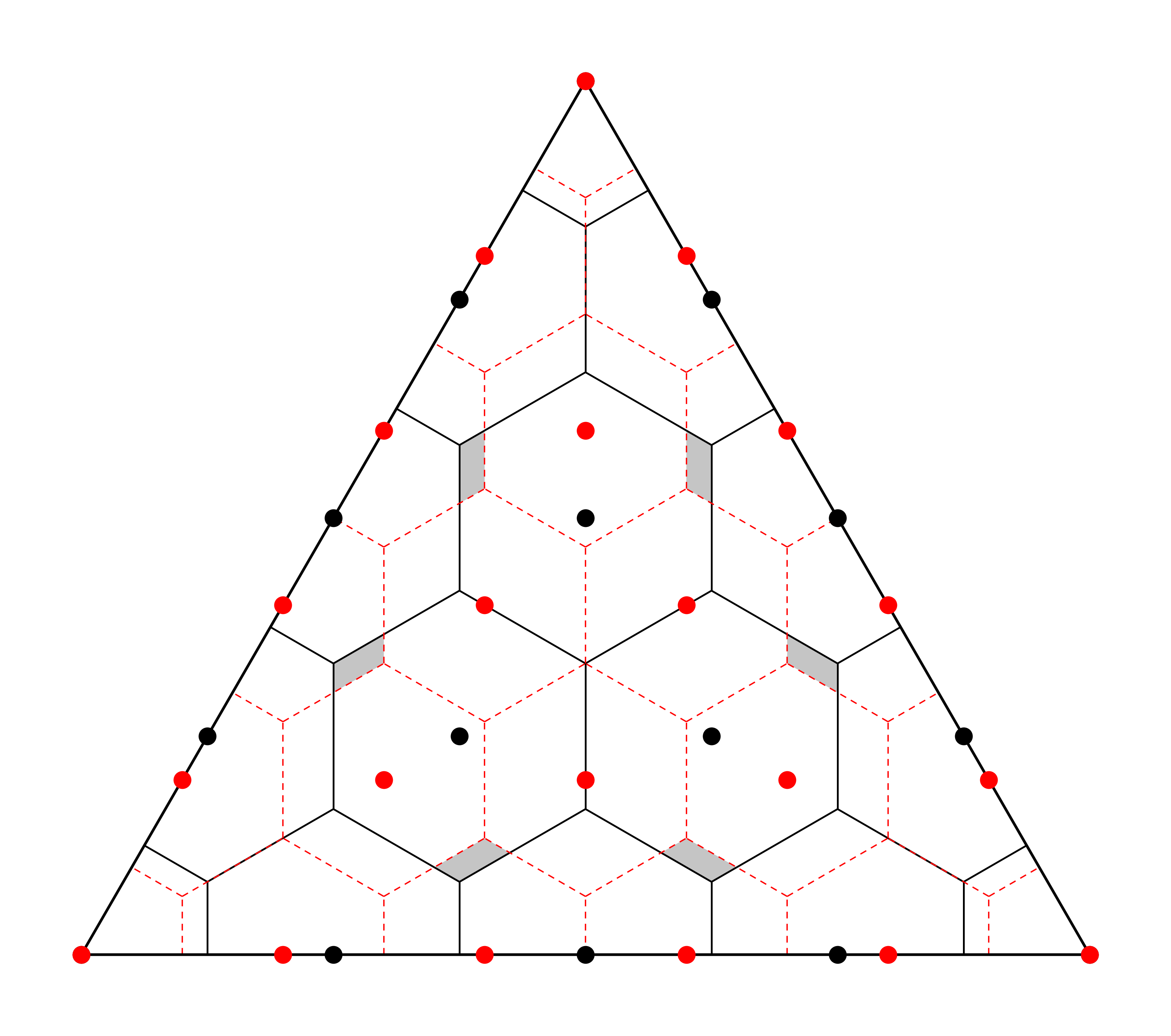}
\end{tabular}
\caption{(Left) The Hamilton apportionment diagram for $H=4$. (Right) The apportionment diagrams for $H=4$ (solid black lines) and $H=5$ (red dashed lines), scaled to the unit simplex $\Delta^2$. The shaded regions correspond to population distributions exhibiting the Alabama paradox when the House size increases from $4$ to $5$. This figure appears in Bradberry's paper \cite{Bradberry}.}
\label{figure:first_HAM}
\end{figure}

\begin{figure}
    \centering\includegraphics[width=0.85\textwidth]{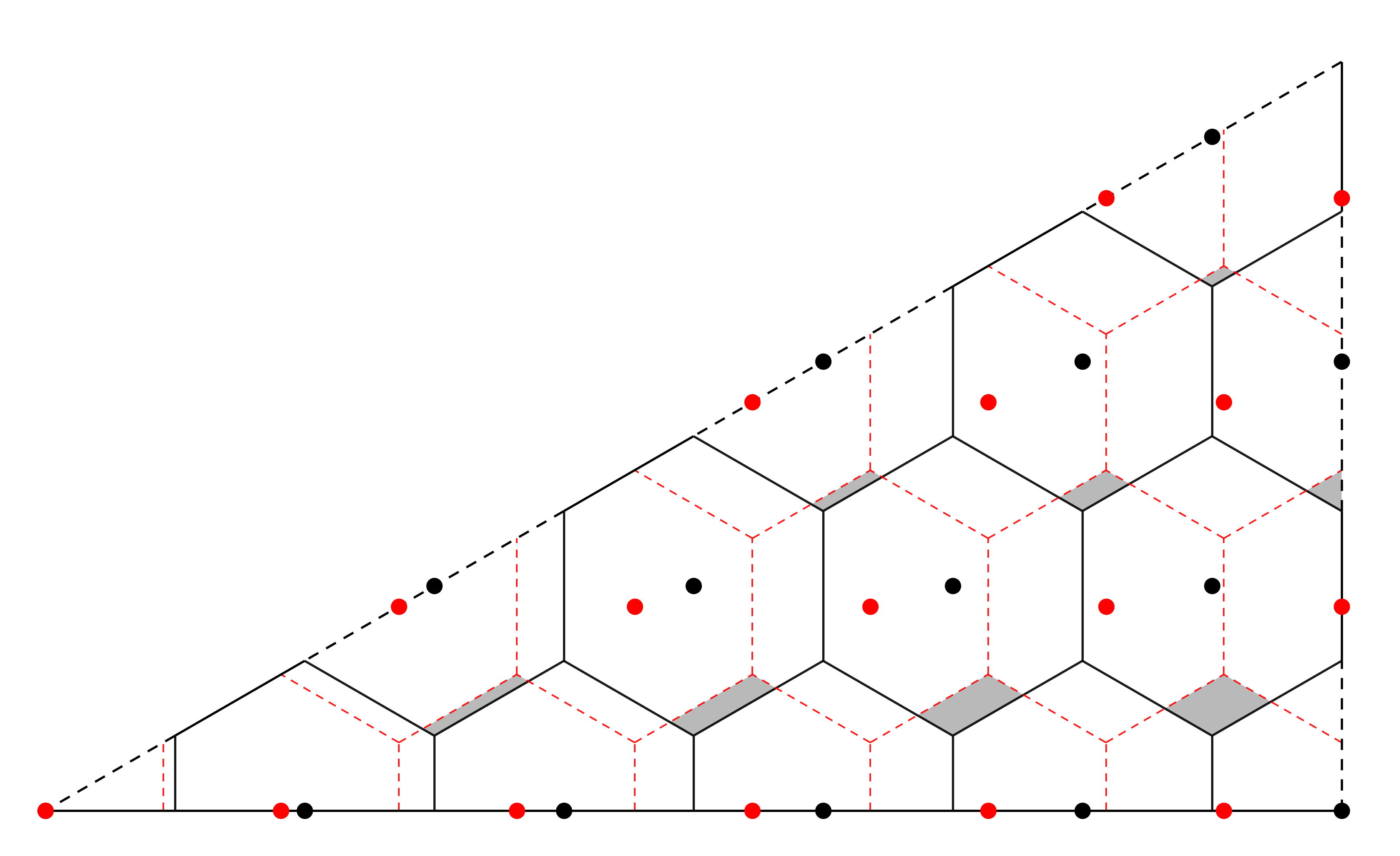}
    \caption{The part of the Hamilton apportionment diagram for $H=10$ (solid black lines) where $p_1\ge p_2 \ge p_3$. The red dashed lines correspond to the $H=11$ diagram, and the shaded regions are the points which demonstrate the Alabama paradox.}
    \label{figure_AL_10}
\end{figure}

\section{Probability of the Alabama Paradox}

What is the probability that the Alabama paradox occurs? To answer this question, we must first specify a probability distribution on the set of possible population proportions. Bradberry assumes a uniform distribution on the simplex $\Delta^2$, so that the probability of the paradox is simply the proportion of the simplex's area corresponding to overlap regions. This probability distribution is often termed the ``impartial anonymous culture model'' and is widely used in apportionment theory or voting theory. We let $p(H)$ denote the probability that a proportion vector demonstrates the Alabama paradox for $H$ seats under Hamilton's method when $S=3$, assuming a uniform distribution on the simplex.

In Theorem 2 of \cite{Bradberry}, Bradberry states:

\begin{quote}
Under the assumption that populations are uniformly distributed over
the apportionment diagram, in the apportionment by Hamilton's method of
$H$ seats among three states, the Alabama paradox has probability of occurrence
\begin{equation}\label{eq:probability}
p(H)=\frac{(H-1)(H-2)}{12H(H+1)}.
\end{equation}
\end{quote}

Bradberry uses the geometry of the Hamilton simplex to generate his formula, which is incorrect. For a fixed $H$, he calculates the areas of all  overlap regions, sums the areas, and divides this sum by the area of the simplex to obtain a probability.

How did I know Equation \ref{eq:probability} was incorrect? While working on the book \cite{JonesMcCuneWilson_book}, I ran Monte Carlo simulations to estimate the probabilities of various apportionment paradoxes in the three-state case. I was particularly interested in the limiting behavior as $H\rightarrow\infty$, so I ran simulations with very large values of $H$. For $H=10,000$, my simulations suggested an Alabama paradox probability of $1/36$, rather than the $1/12$ predicted by Bradberry's formula.  I assumed there was a bug in my code, and I spent hours trying to figure out where I had gone wrong. Eventually, I came across \cite{JansonLinusson}, where the authors prove that $\displaystyle\lim_{H\rightarrow\infty}p(H)=1/36$. My code was vindicated, and I knew Bradberry's formula was incorrect. The work in \cite{JansonLinusson} does not rely on explicit formulas for a fixed $H$, and thus the authors do not determine what Bradberry's formula should be.

Bradberry's geometric approach is sound, and simplicial geometry has been used to study apportionment methods from various angles in the literature. For example, the apportionment simplex has been used to analyze the probability of the so-called elimination paradox in presidential primary apportionment \cite{EliminationParadox,JonesMcCuneWilson_book} and to investigate the extent to which different apportionment methods systematically favor large or small states \cite{SchusterEtAl2003,DrtonSchwingenschloegl2004,SchwingenschloeglDrton2004}. Thus, to determine the correct formula, we revisit his overlapping region technique and obtain the following theorem.

\begin{theorem}\label{main_thm}
Let $H\ge 3$. Suppose that population proportions are uniformly distributed over the
apportionment simplex for three states. Then
\begin{equation*}
p(H)=
\begin{cases}
\displaystyle \frac{(H+3)(H-2)}{36H(H+1)}
& H\equiv 0,2\pmod{3},\\[8pt]
\displaystyle \frac{(H-1)(H+2)(H^2+H-4)}{36H^2(H+1)^2}
& H\equiv 1\pmod{3}.
\end{cases}
\end{equation*}
\end{theorem}

Interestingly, the correct probability agrees with Bradberry's formula only when $H=3$, even though the only visualization of the Alabama paradox in his paper uses $H=4$.

Before diving into the proof we make two observations. First, in the three-state case it is known that only the smallest state can lose a seat when the House size increases (see \cite{JansonLinusson}, for example), even when the size of $H$ increases by more than one. The basic idea behind this observation is that when the House size increases, each state's quota increases by an amount proportional to its population. Thus, a larger state cannot lose a seat while a smaller state does not, since the larger state's quota increases by more. Working out the exact algebra is a fun exercise, left to the reader. As a result, in the proof we assume only the smallest state can lose a seat.

Second, it is well-known that when the House size increases by one, no state can gain more than one seat. This follows from Bradberry's geometric characterization of Hamilton's method: apportionment vectors are the closest lattice points to their corresponding quota vectors. We leave the short verification to the reader.

\begin{proof}
We prove the result in three steps. 

\begin{enumerate}
\item For a fixed apportionment vector $\mathbf{a}=(a_1, a_2, a_3)$, $a_3\ge 1$, with corresponding proportion vector $(p_1, p_2, p_3)$ in the $H$ diagram, we provide proportion-based inequalities which characterize the region where the apportionment regions for $\mathbf{a}$ and $(a_1+1, a_2+1, a_3-1)$ overlap.
\item Using the characterization from (1), we calculate the area of a given overlap region.
\item We sum the areas of the overlap regions and divide by the total area of the simplex to obtain our probability.
\end{enumerate}
Assume the smallest state is state 3 so that $p_3\le p_1,p_2$.

\textbf{Step 1}: Fix an apportionment $\mathbf{a}=(a_1,a_2,a_3)$ for $H$ seats, where $a_3\ge 1$. Suppose state 3 loses a seat when the House size increases to $H+1$. The corresponding apportionment must be $\mathbf{a}'=(a_1+1,a_2+1,a_3-1)$  under Hamilton's method since no state can gain more than one seat when the House size increases by one. We must find the area of the intersection of the regions corresponding to $\mathbf{a}$ and $\mathbf{a}'$, after scaling both apportionment diagrams to $\Delta^2$.

To describe this intersection, we exploit Bradberry's observation that a quota vector $\mathbf{q}$ is mapped to the closest apportionment vector under the standard Euclidean metric. Suppose $(p_1,p_2,p_3)$ is a proportion vector such that the corresponding quota vector $\mathbf{q}=(Hp_1,Hp_2,Hp_3)$ maps to $\mathbf{a}$. Since $\mathbf{a}$ is the closest apportionment vector to $\mathbf{q}$, it must be at least as close to $\mathbf{q}$ as its neighboring apportionment vector $\mathbf{b}=(a_1+1,a_2,a_3-1)$. Thus,
\[
||\mathbf{q}-\mathbf{a}||\le ||\mathbf{q}-\mathbf{b}||.
\]
Expanding and canceling common terms gives
\[
(q_1-a_1)^2+(q_3-a_3)^2
\le (q_1-a_1-1)^2+(q_3-a_3+1)^2,
\]
which simplifies to $q_1-q_3=H(p_1-p_3)\le a_1-a_3+1.$

Comparing $\mathbf{a}$ to the neighboring vector $(a_1-1,a_2,a_3+1)$ similarly gives
\[
H(p_1-p_3)\ge a_1-a_3-1.
\]
Combining these inequalities gives
\begin{equation}\label{eq:first_inequality}
\frac{a_1-a_3-1}{H}\le p_1-p_3\le
\frac{a_1-a_3+1}{H}.
\end{equation}

 Since $(a_1+1)-(a_3-1)=a_1-a_3+2,$
applying the same argument to $\mathbf{a'}$ in the $(H+1)$-seat simplex gives

\begin{equation}\label{eq:second_inequality}
\frac{a_1-a_3+1}{H+1}\le p_1-p_3\le
\frac{a_1-a_3+3}{H+1}.
\end{equation}
Intersecting the two intervals given by Inequalities \ref{eq:first_inequality} and \ref{eq:second_inequality} gives the following bounds for $p_1-p_3$:
\begin{equation}\label{first_bounds}
\frac{a_1-a_3+1}{H+1}\le p_1-p_3\le
\frac{a_1-a_3+1}{H}.\end{equation}

We obtained Inequality \ref{first_bounds} by analyzing the neighboring apportionment vector $(a_1+1, a_2, a_3-1)$. Repeating the same argument for the neighboring apportionment $(a_1, a_2+1, a_3-1)$ gives
\begin{equation}\label{second_bounds}
\frac{a_2-a_3+1}{H+1}\le p_2-p_3\le
\frac{a_2-a_3+1}{H}.\end{equation}

Inequalities \ref{first_bounds} and \ref{second_bounds} are necessary conditions for the overlap region to have nonzero area. To establish sufficiency, we must also consider the remaining nearest-neighbor conditions, which involve $p_1-p_2$. Since $p_1-p_2=(p_1-p_3)-(p_2-p_3)$, it is straightforward to verify that these conditions are already implied by our previous inequalities. Thus, an analysis of the remaining nearest apportionment neighbor imposes no additional restrictions, and Inequalities \ref{first_bounds} and \ref{second_bounds} completely characterize the overlap regions. Moreover, since $a_1,a_2\ge a_3\ge 1$, the resulting overlap regions lie entirely within the simplex and are not truncated by its boundary.

\textbf{Step 2}: Now that we have inequalities characterizing the overlap regions, we calculate their areas. To do this, we use a change of variables to transform the regions into rectangles, making the areas easy to calculate.

Define $x=p_1-p_3$ and $y=p_2-p_3$.
Since $p_1+p_2+p_3=1$, we have $p_3=(1-x-y)/3$, and consequently
\[
p_1=\frac{1+2x-y}{3}
\ \ \text{and} \ \ 
p_2=\frac{1-x+2y}{3}.
\]
The Jacobian determinant of this transformation is
\[
\left|
\frac{\partial(p_1,p_2)}{\partial(x,y)}
\right|
=
\begin{vmatrix}
2/3 & -1/3\\
-1/3 & 2/3
\end{vmatrix}
=\frac13.
\]

Thus, the area of a region in the unit simplex
is one-third of the area of its transformation in the $(x,y)$-plane. By Inequalities \ref{first_bounds} and
\ref{second_bounds}, under our change of variables the overlap region becomes
the rectangle
\[
\frac{a_1-a_3+1}{H+1}\le x\le\frac{a_1-a_3+1}{H},
\qquad
\frac{a_2-a_3+1}{H+1}\le y\le\frac{a_2-a_3+1}{H}.
\]
Its side lengths are
\[
\frac{a_1-a_3+1}{H(H+1)}
\qquad\text{and}\qquad
\frac{a_2-a_3+1}{H(H+1)}.
\]
Therefore, the area of the corresponding
overlap region in the $(p_1,p_2)$-plane is
\begin{equation}\label{eq:parallelogram_area}
A_{\mathbf a}
=
\frac{(a_1-a_3+1)(a_2-a_3+1)}
{3H^2(H+1)^2}.
\end{equation}

\textbf{Step 3:} We now sum the areas of the overlap
regions and divide by the area of the simplex.
Since state 3 is the smallest state, write
\[
a_3=t,\qquad a_1=t+i,\qquad a_2=t+j,
\]
where $1\le t\le\lfloor H/3\rfloor$ and
$i,j\ge 0$ satisfy $i+j=H-3t$.

By Equation \ref{eq:parallelogram_area}, the total
area corresponding to state 3 losing a seat is
\[
\frac{1}{3H^2(H+1)^2}
\sum_{t=1}^{\lfloor H/3\rfloor}
\sum_{i=0}^{H-3t}
(i+1)(H-3t-i+1)
=
\frac{1}{3H^2(H+1)^2}
\sum_{t=1}^{\lfloor H/3\rfloor}
\binom{H-3t+3}{3},
\]
where the equality comes from using the identity
$\sum_{i=0}^{k}(i+1)(k-i+1)=\binom{k+3}{3}$.

By symmetry, the total area of all paradox regions
is three times this quantity (multiplication by three accounts for the other cases where state 1 or state 2 is smallest). Since the unit
simplex has area $1/2$, the desired probability is
\begin{equation}\label{eq:probability_sum}
p(H)=
\frac{2}{H^2(H+1)^2}
\sum_{t=1}^{\lfloor H/3\rfloor}
\binom{H-3t+3}{3}.
\end{equation}

Writing $H=3r+s$, where $s\in\{0,1,2\}$,
the sum becomes
\[
\sum_{t=1}^{\lfloor H/3\rfloor}
\binom{H-3t+3}{3}
=
\begin{cases}
\displaystyle
\frac{r(r+1)(3r-2)(3r+1)}{8},
& s=0,\\[8pt]
\displaystyle
\frac{r(r+1)(9r^2+9r-2)}{8},
& s=1,\\[8pt]
\displaystyle
\frac{r(r+1)(3r+2)(3r+5)}{8},
& s=2.
\end{cases}
\]
Substituting into Equation \ref{eq:probability_sum}
and simplifying gives the result. \end{proof}

\section{Bradberry's Error}

Where did Bradberry go wrong? How is his limiting probability off by a factor of 3, and how does his formula not have any congruence modulo 3?

It is impossible to know the source of his error with certainty because his proof is essentially just one sentence. Before stating his theorem, he says, ``It is straightforward, although tedious, to calculate the areas of these $[$overlap$]$ regions.'' He does not present a formula for a given region's area, he does not provide the bounds on the rectangle after the change of variables, etc. However, the lack of details presents a fun challenge, and often figuring out others' mistakes can be mathematically enlightening. So let's investigate. 

Initially, because I was examining the case of $H\rightarrow \infty$, I thought he had just made an accounting error and not divided by  3 at some point. That is, initially I thought his formula was correct if we replace a 12 with a 36 in Equation \ref{eq:probability}. This would be the case if he forgot to adjust by the Jacobian of $1/3$ under the change of variables, for example (a mistake which might resonate with many multivariable calculus students). Of course, Theorem \ref{main_thm} shows this is wrong. 

I then thought that maybe he was misled by the image in the right of Figure \ref{figure:first_HAM} into thinking that all of the overlapping parallelograms have the same area.  If he used an area of $\frac{1}{36H(H+1)}$ for each overlap region in the unit simplex then the summation would yield his formula.  However, I find this explanation unlikely. Given Bradberry's familiarity with the geometry, surely he would know that for larger House sizes (such as $H=10$; see Figure  \ref{figure_AL_10}) the overlap regions in general do not have the same area.

After considering many other candidate explanations, I think the following error is most likely. Bradberry seems familiar enough with the geometry to know the correct area formula given in Equation \ref{eq:parallelogram_area}, and some of his exposition suggests he used the change of variables $x=p_1-p_3, \ y=p_2-p_3$  to calculate it (as we did). If this is the case, he must have incorrectly summed the areas. Recall that in our proof, in order to nicely write our summation we used $a_3=t$, $a_1=t+i$, and $a_2=t+j$. Furthermore, $1\le t\le\lfloor H/3\rfloor$ and
$i,j\ge 0$ satisfy $i+j=H-3t$. Why do we need the constraint $i+j=H-3t$? Since state 3 has the fewest seats, if we assign $t$ seats to state 3 then we must also assign $t$ seats to the other two states, leaving $H-3t$ seats remaining to assign. Thus, $i$ (respectively $j$) is the number of seats assigned to state 1 (respectively state 2) in excess of $t$.

Suppose Bradberry correctly calculated the individual parallelogram areas but made a mistake when indexing them. What if Bradberry treated $i$ and $j$ as though they indexed ordinary positive seat allocations, rather than the numbers of seats held by states 1 and 2 in excess of state 3? Since we are considering interior apportionment lattice points in the simplex, each state must receive at least one seat. Thus, he might have reserved $t$ seats for state 3, but only one seat each for states 1 and 2, leaving $H-t-2$ seats to distribute. This would lead to the incorrect constraint $i+j=H-t-2$, rather than $i+j=H-3t$. Under this mistaken constraint, $t$ would range from $1$ to $H-2$, rather than stopping at $\lfloor H/3\rfloor$. The resulting calculation would be
\[
\begin{aligned}
p(H)
&=\frac{2}{H^2(H+1)^2}
\sum_{t=1}^{H-2}
\sum_{i=0}^{H-t-2}
(i+1)(H-t-i-1)\\
&=\frac{2}{H^2(H+1)^2}
\sum_{t=1}^{H-2}
\binom{H-t+1}{3}\\
&=\frac{(H-1)(H-2)}{12H(H+1)}.
\end{aligned}
\]
Thus, this mistake would cause Bradberry to count many nonexistent parallelograms while still using the correct formula for their supposed areas, resulting in his formula from Equation \ref{eq:probability}. In particular, this mistake would explain why his theorem does not depend on congruence modulo 3. Ultimately, I can't prove this is what he did, but the error is natural and does recover his theorem.


\section{Larger House size increases}

Previously we used the geometry of overlapping apportionment regions to calculate the probability of the Alabama paradox. We now use the same geometry to investigate a different question: how much can the House size increase while still allowing the paradox to occur? Could a state lose a seat even if we double the size of the House, for example? 

\begin{example}
Suppose $p_1=p_2=0.4845$ and $p_3=0.031$. For a House size of $H=11$, the quotas are $q_1=q_2=5.3295$ and $q_3=0.341$, producing the Hamilton apportionment $(5,5,1)$. If we increase the House size to 20 seats, the quotas become $q_1=q_2=9.69$ and $q_3=0.62$, producing the Hamilton apportionment $(10,10,0)$.
\end{example}

Thus, a state can lose all of its representation even when the size of the House nearly doubles.

As a historical side note, the Alabama paradox can produce some rather strange behavior as the House size increases. A state's allocation can oscillate between two values, gaining and losing a seat repeatedly. This phenomenon inspired one of my favorite quotations about the mathematics of apportionment and the dastardly nature of the Alabama paradox.

Following the 1900 census, the House Select Committee on the Twelfth Census examined Hamilton apportionments for House sizes ranging from 350 to 400. Over this range, Maine's allocation repeatedly switched between three and four seats. The committee proposed fixing the House size at 357, which would give Maine only three seats. During the ensuing debate, Representative John Littlefield of Maine expressed a frustration with mathematics that might resonate with some students:

\begin{quote}
Maine loses on 382. She loses again with 386, and does not lose with 387 or 388. Then she loses again on 389 and 390 ... [I]t does seem as though mathematics and science had combined to make a shuttlecock and battledore of the state of Maine in connection with the scientific basis upon which this bill is presented ... In Maine comes and out Maine goes ... God help the state of Maine when mathematics reach for her and undertake to strike her down! \cite{BalinskiYoung}
\end{quote}

Returning to the mathematics (which hopefully won't strike us down), we turn to examining the largest increase in House size for which the Alabama paradox can occur with positive probability. Let $p_k(H)$ denote the probability that some state loses a seat when the House size increases from $H$ to $H+k$, assuming a uniform distribution on the  simplex. Thus, $p_1(H)=p(H)$, where $p(H)$ is the probability studied in Theorem 1. Define
\[
I_{\max}(H)=\max\{k\geq 1:p_k(H)>0\}.
\]

We obtain the following bound.

\begin{proposition}\label{prop:max_increase}
Let $H\geq 3$ and $S=3$. Then $I_{max}(H)\le H-2$. Furthermore, if $H$ is odd then $I_{max}(H)=H-2$.
\end{proposition}

Thus, we cannot double the size of the House and obtain a paradox, but for odd $H$ we can nearly double.

\begin{proof}
First we show that if $H$ is odd, there exists a proportion vector demonstrating the paradox when we increase from $H$ to $2H-2$ seats. Let $\epsilon>0$ such that $0<\epsilon<\frac{1}{3(H-1)}$. Consider the choice of proportions \[p_3=\frac{\epsilon+1/3}{H}, \ \ \  p_1=p_2=\frac{1-p_3}{2}.\] For a House size of $H$ the quotas are \[q_1=q_2=\frac{H-1}{2}+\frac{1}{3}-\frac{\epsilon}{2}, \ \ \ q_3=\frac{1}{3}+\epsilon\] and for a House size of $2H-2$ the quotas are \[q_1'=q_2'=H-\frac{4}{3}+\frac{1}{3H}-\left(1-\frac{1}{H}\right)\epsilon, \ \ \ q_3'=\frac{2}{3}-\frac{2}{3H}+\left(2-\frac{2}{H}\right)\epsilon.\]
Since $H$ is odd, the fractional part of $q_1$ and $q_2$ is $1/3-\epsilon/2$ and the fractional part of $q_3$ is $1/3+\epsilon$, and thus state 3 receives one seat when the House size is $H$.

To analyze the fractional parts of the quotas when the House size is $2H-2$, note that $q_1'=q_2'=H-1-q_3'/2$ and thus the fractional part for states 1 and 2 is $1-q_3'/2$ while the fractional part for state 3 is $q_3'$. Since $q_3'<2/3$, the fractional part for states 1 and 2 is larger than the fractional part for state 3, giving an allocation of $(H-1, H-1, 0)$, and state 3 loses its seat.

Now suppose $H\ge 3$ is arbitrary and let $k >H-2$. We show that no state loses a seat when increasing the House size from $H$ to $H+k$. The proof analyzes the apportionment diagrams after mapping them to 2-space so that we have concrete $x$ and $y$ coordinates for vertices of the hexagons.

We use the coordinate transformation
\[
x=p_2+\frac{p_3}{2},
\qquad
y=\frac{\sqrt{3}}{2}p_3,
\]
which is a standard mapping of the simplex to 2-space \cite{ChambersWallMcCune2024,CutoffParadox,JonesMcCuneWilson_book}. In particular, this mapping is used to create the images for this article. Because the mapping is standard, we do not fully justify each geometric claim about coordinates of points, or slopes and equations of lines. 

For an allocation $(a_1,a_2,a_3)$ with House size $H$
and an allocation $(b_1,b_2,b_3)$ with House size $H+k$,
the bottom vertex of the former hexagon and the top vertex
of the latter have heights
\[
y_{\mathrm{bottom}}=\frac{3a_3-2}{2\sqrt{3}H},
\ \ \ \text{and} \ \ \ 
y_{\mathrm{top}}=\frac{3b_3+2}{2\sqrt{3}(H+k)},
\]
respectively.

Suppose that $a_3\ge 2$ and $b_3<a_3$. Since
$b_3\le a_3-1$ and $k\ge H-1$, we have
\[
y_{\mathrm{top}}
\le \frac{3a_3-1}{2\sqrt{3}(2H-1)}
< \frac{3a_3-2}{2\sqrt{3}H}
=y_{\mathrm{bottom}}.
\]

In other words, the bottom vertex of the $(a_1, a_2, a_3)$ hexagon is higher than the top vertex of the $(b_1, b_2, b_3)$ hexagon. Thus in this case the two hexagons cannot have overlapping interiors, and the paradox is not possible. Consequently,
a loss of a seat is possible only if $a_3=1$ and $b_3=0$.

In this remaining case, the bottom vertex of the original
hexagon and the top vertex of the new hexagon have heights
\[
y_{\text{bottom}}=\frac{1}{2\sqrt{3}H} \ \ \ \text{and} \ \ \ 
y_{\text{top}}=\frac{1}{\sqrt{3}(H+k)},
\]
respectively.
If $k\ge H$, the latter is no greater than the former,
so the hexagons again cannot have overlapping interiors. Figure \ref{figure:Hamilton_H10_to_20} illustrates this observation for $H=k=10$. If $k=H$ then there are only boundary intersections with zero overlap area.

It therefore remains only to consider the case where $k=H-1$, $a_3=1$, and $b_3=0$. We must show that no hexagon with $a_3=1$ for House size $H$ has an interior intersection with a hexagon with $b_3=0$ for House size $2H-1$. In this case the bottom of the hexagon for $(a_1, a_2, 1)$ is lower than the top vertex of the half-hexagon for $(b_1, b_2, 0)$, so we can't simply use the vertex heights to justify a lack of intersection. The bottom of Figure \ref{figure:Hamilton_H10_to_20} illustrates this overlap in vertical ranges for $H=10$ using a blue strip: the bottom blue line occurs at the height of the bottom of the black hexagons, while the top blue line occurs at the height of the top of the red hexagons. Although the vertical ranges of these hexagons overlap, their horizontal positions and sloping boundaries do prevent an intersection.

To justify the lack of intersection in this case, consider the bottom vertex of a black hexagon (corresponding to $a_3=1$) and the top vertex of a red half-hexagon (corresponding to $b_3=0$). Their horizontal coordinates are
$x_B=(a_2+1/2)/H$ and $x_R=b_2/(2H-1)$, respectively.
We show that the red top vertex cannot lie inside
the black hexagon.

The vertical distance between these two vertices is
$1/(2\sqrt{3}H(2H-1))$. Since the lower edges of the
black hexagon have slopes $\pm 1/\sqrt{3}$, moving
upward from its bottom vertex to the height of the
red top vertex corresponds to moving horizontally
by $1/(2H(2H-1))$ along either edge. Thus, at this
height, the black hexagon extends horizontally from
$x_B-1/(2H(2H-1))$ to $x_B+1/(2H(2H-1))$.

However, the horizontal distance between the two
vertices satisfies
\[
|x_B-x_R|
=\frac{|(2H-1)(2a_2+1)-2Hb_2|}{2H(2H-1)}
\geq \frac{1}{2H(2H-1)},
\]
since the numerator is an odd integer. Therefore,
the red top vertex lies on or outside the boundary
of the black hexagon. Moreover, the sloping edges descending from the red top vertex are parallel to the lower edges of the black hexagon on the opposite side. Thus, as we move downward, the red half-hexagon cannot cross into the interior of the black hexagon.
\end{proof}

The geometry not only provides a proof, but also helps explain why other proof approaches can be difficult. Before developing this geometric proof, I attempted a purely algebraic approach by analyzing the fractional parts of the quotas after increasing the House size. The algebra worked out nicely for $k\geq H$, but was a huge mess for $k=H-1$. The geometry illustrates why: when $k=H-1$, the vertical ranges of the relevant hexagons overlap, whereas they don't for $k\geq H$.

\begin{figure}
 \centering   \includegraphics[width=0.75\textwidth]{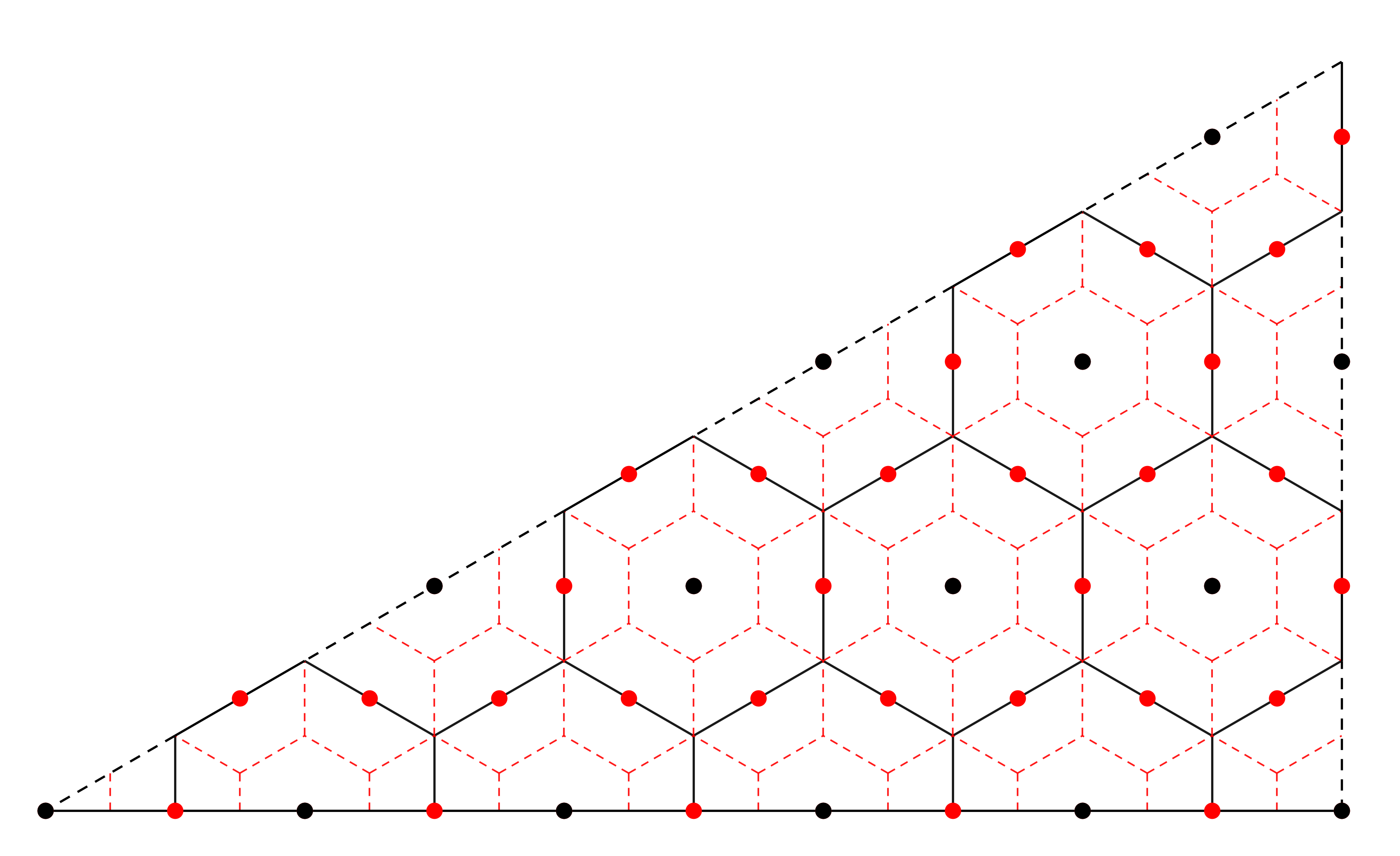}

 \vspace{5pt}
 
 \includegraphics[width=0.85\textwidth]{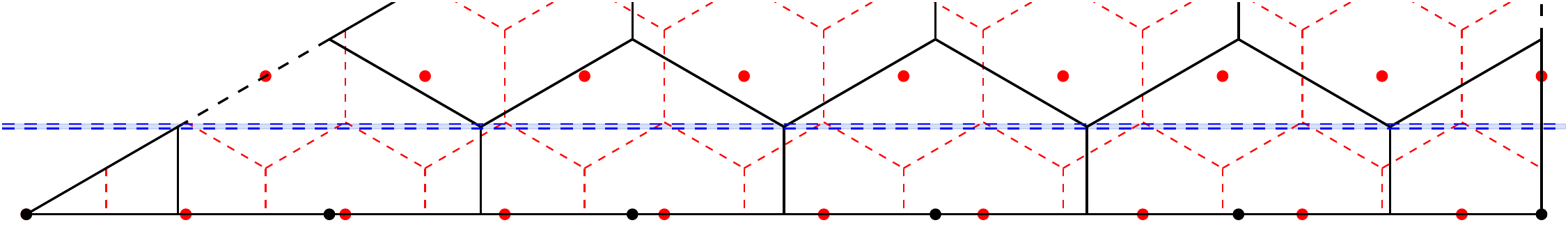}
    \caption{(Top) The Hamilton apportionment diagrams for $H=10$ (solid black lines) and $H=20$ (red dashed lines), where $p_1\ge p_2\ge p_3$. (Bottom) The lower piece of the apportionment diagram for $H=10$ and $H=19$. The blue strip shows the height range overlap between the lower row of red half-hexagons and the first row of full black hexagons.}
    \label{figure:Hamilton_H10_to_20}
\end{figure}

\section{Further Explorations}
Proposition \ref{prop:max_increase} leaves open the value of $I_{max}(H)$ when $H$ is even. I've experimented with several small even values of  $H$ and I'm unsure of the exact formula. The value of $I_{max}(H)$ tends to bounce between $H-3$, $H-4$, and $H-5$, and there seems to be some kind of dependence modulo 12. I have only considered even $H$ values up to 40, so any stab at a general formula is pure guesswork.

When I first started playing around with the $I_{max}(H)$ values, I used Python scripts to give me the exact overlap areas when increasing $H$ to $H+k$ for various $k$, thereby getting values of $p_k(H)$. I used the Python package Shapely because it can numerically calculate the areas of the overlap regions for a fixed $H$. For $H=10$, the script produced the following values: 

\begin{center}
\begin{tabular}{c|ccccc}
\hline
$k$ & 1 & 2 & 3 & 4 & 5\\
\hline
$p_k(H)$ & 2.6\% & 1.1\% & 0.7\% & 0.4\% & 0\%\\
\hline
\end{tabular}
\end{center}

Upon seeing the 0\% for $k=5$, I thought I had discovered $I_{max}(10)=4$. I then noticed that for other even values of $H$, the value $k=H/2$ produced a probability of zero, leading me to conjecture that $I_{max}(H)=H/2-1$ for even $H$. 

\begin{figure}
\centering    \includegraphics[width=0.75\textwidth]{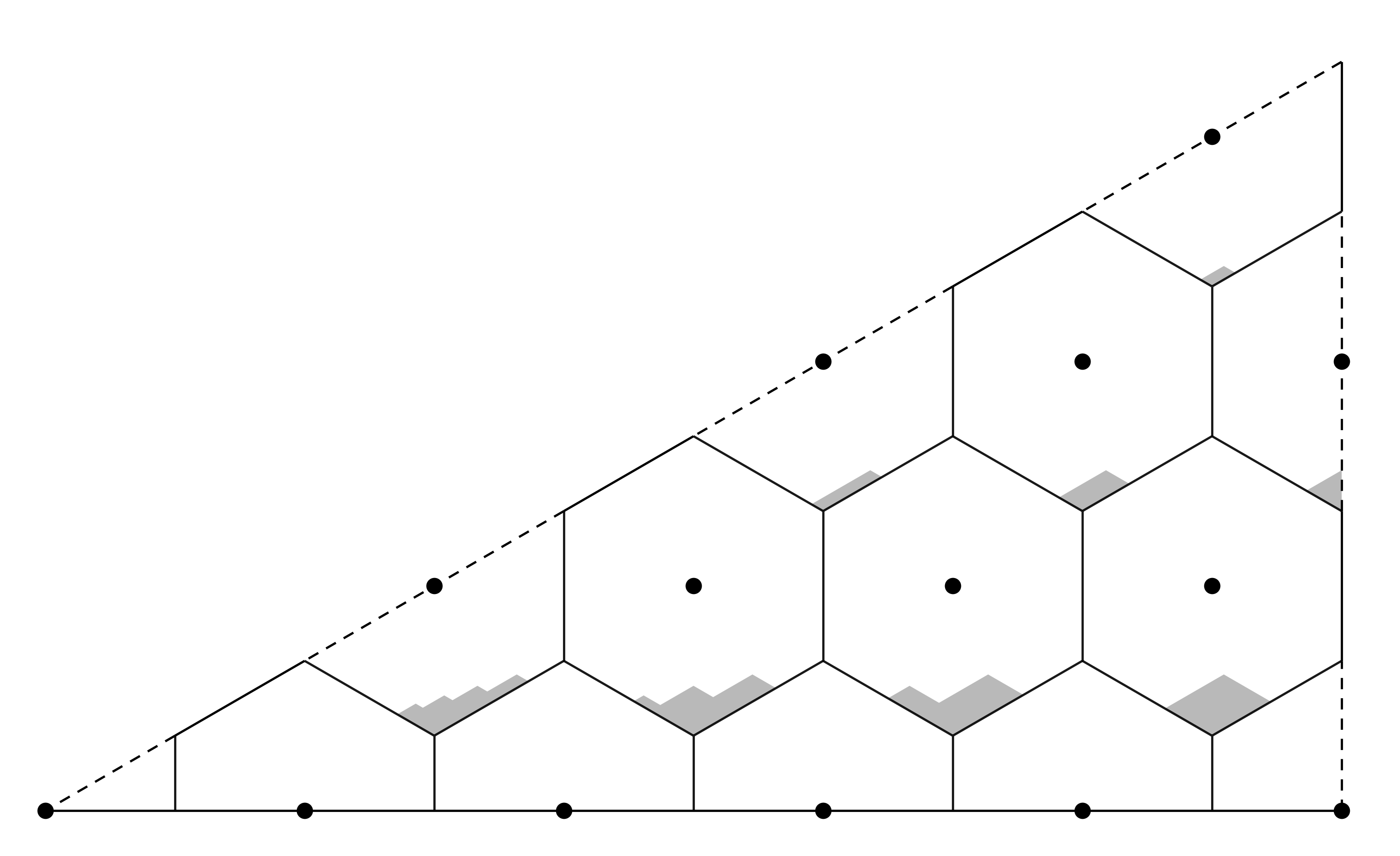}
    \caption{The part of the Hamilton apportionment diagram for $H=10$ where $p_1\ge p_2 \ge p_3$. The shaded regions represent proportion vectors where some increase in House size causes the smallest state to lose a seat.}
    \label{figure_AL_10_all_increases}
\end{figure}

However, when trying to understand better the structure of the diagrams for larger $k$ values, I produced overlap images for $H=10$ and $k\in\{6,7\}$. I was surprised to see probabilities of 0.06\% and 0.03\% for $k=6$ and $k=7$, respectively. Both of these probabilities are well-known to be greater than zero; I then realized I had been assuming $p_k(H)$ was a non-increasing function in $k$, and I had made this assumption without consciously expressing it. To me it seemed self-evident these probabilities would decrease because as $k$ increases, it becomes ``harder'' for the Alabama paradox to occur. There is a deep satisfaction, known to most mathematicians, in being surprised by a result (even when the surprise means we were initially wrong). After noticing the sometimes-increasing nature of $p_k(H)$, I played around with many values of $H$ and $k$, leading to the following conjecture:

\begin{conjecture}
If $H$ is odd then the sequence of probabilities $p_k(H)$ for $k\ge1$ is non-increasing. If $H$ is even then $p_{H/2}(H)=0$ while $p_{H/2+1}(H)>0$. In particular, when $H$ is even it seems that $p_k(H)$ decreases for $1\le k \le H/2$, then is decreasing again for $H/2+1\le k\le I_{max}(H)$, and thus experiences the only increase when moving from $k=H/2$ to $k=H/2+1$.
\end{conjecture}

Finally, there are other ways to define what it means for a proportion vector to demonstrate the Alabama paradox. The original definition is restrictive in that it only increases the House size by one. The most expansive definition would be to say a vector demonstrates the paradox if some increase of House size causes a state to lose  a seat. Let $p_{total}(H)$ denote the probability that there exists a House size $H'>H$ for which a state loses a seat when moving from $H$ to $H'$ seats for a randomly chosen proportion vector. An analysis of this probability would embrace the broadest possible understanding of what it means for a proportion vector to demonstrate the paradox. 

The geometry suggests that a formula for $p_{total}(H)$ would be messy and quite difficult to calculate. Figure \ref{figure_AL_10_all_increases} shows the shaded regions corresponding to this probability for $H=10$. The overlap regions are in general no longer nice, but instead sometimes have jagged ridges formed by the unions of parallelograms. Even though the geometry is daunting, perhaps there is some algebraic or geometric wizardry that could produce a closed formula in terms of $H$.

\section{Conclusion}

I love apportionment theory because at first glance it seems like  a simple problem of elementary-school rounding, but the mathematics runs much deeper, often in surprising ways. Ultimately, I think the lesson of this article is something I constantly stress to my wife: increasing the size doesn't always leave everyone satisfied, so perhaps we should just appreciate the size we have, which is perfectly adequate.


%

\section*{Acknowledgments} Thanks to Courtney Gibbons for making me think about apportionment again. 

\bibliographystyle{vancouver}
\bibliography{VancouverExamples}

\begin{abstract}
The Alabama paradox occurs when increasing the total number of seats in a legislature causes a state to lose a seat. We revisit a geometric analysis of this paradox published by Brent Bradberry in 1992, correct an erroneous probability formula, investigate the source of the error, and explore further consequences of the geometry.
\end{abstract}


\end{document}